\documentclass[12pt, reqno]{amsart}
\usepackage{amsmath, amsthm, amscd, amsfonts, amssymb, graphicx, color}
\usepackage[bookmarksnumbered, colorlinks, plainpages]{hyperref}
\hypersetup{colorlinks=true,linkcolor=red, anchorcolor=green, citecolor=cyan, urlcolor=red, filecolor=magenta, pdftoolbar=true}

\newtheorem{theorem}{Theorem}[section]
\newtheorem{lemma}[theorem]{Lemma}

\theoremstyle{definition}
\newtheorem{definition}[theorem]{Definition}
\newtheorem{example}[theorem]{Example}

\theoremstyle{remark}
\newtheorem{remark}[theorem]{Remark}
\numberwithin{equation}{section}

\begin{document}

\setcounter{page}{1}

\title[Decay estimates ...]{Decay estimates for the time-fractional evolution equations with time-dependent coefficients}

\author[A. G. Smadiyeva, B. T. Torebek]{Asselya G. Smadiyeva, Berikbol T. Torebek$^*$}

\address{\textcolor[rgb]{0.00,0.00,0.84}{Asselya G. Smadiyeva \newline Al--Farabi Kazakh National University \newline Al--Farabi ave. 71, 050040, Almaty, Kazakhstan \newline and \newline Institute of Mathematics and Mathematical Modeling \newline 125 Pushkin str.,
050010 Almaty, Kazakhstan}}
\email{\textcolor[rgb]{0.00,0.00,0.84}{smadiyeva@math.kz}}

\address{\textcolor[rgb]{0.00,0.00,0.84}{Berikbol T. Torebek \newline Department of Mathematics: Analysis,
Logic and Discrete Mathematics \newline Ghent University, Krijgslaan 281, Ghent, Belgium \newline and \newline Institute of
Mathematics and Mathematical Modeling \newline 125 Pushkin str.,
050010 Almaty, Kazakhstan}}
\email{\textcolor[rgb]{0.00,0.00,0.84}{berikbol.torebek@ugent.be}}

%\dedicatory{This paper is dedicated to Professor ABCD}
%\thanks{This research has been funded by the Science Committee of the Ministry of Education and Science of the Republic of Kazakhstan (Grant No. AP08052046). The research of Torebek is financially supported by the FWO Odysseus 1 grant G.0H94.18N: Analysis and Partial Differential Equations. No new data was collected or generated during the course of research.}

\let\thefootnote\relax\footnote{$^{*}$Corresponding author}



\begin{abstract} In this paper, the initial-boundary value problems for the time-fractional degenerate evolution equations are considered. Firstly, in the linear case, we obtain the optimal rates of decay estimates of the solutions. The decay estimates are also established for the time-fractional evolution equations with nonlinear operators such as: p-Laplacian, the porous medium operator, degenerate operator,  mean curvature operator, and Kirchhoff operator. At the end, some applications of the obtained results are given to derive the decay estimates of global solutions for the time-fractional Fisher-KPP-type equation and the time-fractional porous medium equation with the nonlinear source.
\end{abstract}
\maketitle
\tableofcontents
\section{Introduction}
\subsection{Statement of the problem}
The main purpose of this paper is to study the following time-fractional partial differential equation
\begin{equation}\label{1.1} \partial^{\alpha}_{0+,t} u(t, x) - a(t)\mathcal{A}(u(t, x)) = 0,\,\left({t, x}
\right) \in \mathbb{R}_+\times\Omega:=\Omega_+,\end{equation}
with the Cauchy data
\begin{equation}\label{1.2}
u(0,x)=u_0(x),\,x\in\Omega,
\end{equation}
and with the Dirichlet boundary condition
\begin{equation}\label{1.3}
u(t, x)=0,\,\,t\geq 0,\,x\in\partial\Omega,
\end{equation}
where $0<\alpha\leq 1,\,a\in L^1(\mathbb{R}_+),$ $\Omega\subset\mathbb{R}^n$ is a bounded domain with smooth boundary $\partial\Omega$, and $\partial^{\alpha}_{0+,t}$ is the Caputo fractional derivative \cite[P. 97]{1} of order $\alpha$, such that
$$\partial_{0+,t}^\alpha u(t, x) = \left\{\begin{array}{l}\frac{1}{{\Gamma \left(1- \alpha \right)}}\int\limits_0^t {\left(
{t - s} \right)^{-\alpha} \partial_s u\left(s, x\right)} ds,\,\,\quad\text{if}\,\,\,\,\alpha\in(0,1),\\{}\\ \partial_tu(x,t)=\frac{\partial u}{\partial t}(t, x),\,\,\qquad\qquad\,\,\,\quad\text{if}\,\,\,\,\alpha=1\end{array}\right.$$ and $\mathcal A(u)$ is one of the following linear and nonlinear operators:

\begin{itemize}
\item Laplace operator: $$\mathcal A(u):=\Delta u=\sum\limits_{j=1}^n\frac{\partial^2u}{\partial x_j^2};$$
\item $p$-Laplace operator: $$\mathcal A(u):=\Delta_p u=div\left(|\nabla u|^{p-2}\nabla u\right);$$
\item Porous medium operator: $$\mathcal A(u):=\nabla(g(u)\nabla u);$$
\item Degenerate elliptic operator: $$\mathcal A(u):=f(u)\Delta u;$$
\item  Mean curvature operator: $$\mathcal A(u):=div\left(\frac{\nabla u}{\sqrt{1+|\nabla u|^2}}\right);$$
\item Kirchhoff operator: $$\mathcal A(u):=M\left(\|\nabla u\|_{L^q}\right)\Delta_p u.$$
\end{itemize}

Everywhere below, when we consider the case $\alpha\in(0,1)$, we will assume that the following hypothesis holds:
\begin{description}
\item[(H)] $a(t)\geq \kappa t^\beta,\, \beta>-\alpha,\,\kappa>0.$
\end{description}
Assuming the existence of the considered linear and nonlinear problems, we are interested in the asymptotic behavior of the solutions.
\subsection{Historical background}
The linear equation of the form \eqref{1.1} with time-fractional derivative is called the sub-diffusion equation. This kind of equation describes the slow diffusion (see, for example \cite{Vald3, Hilfer, Main, Metz1, Tar, Uch}). In the case when $\alpha = 1/2$, the equation was interpreted by Nigmatullin \cite{Nig} within a percolation (pectinate) model.

There are a huge number of papers devoted to the study of the existence and uniqueness of the solutions (local or global) of initial-boundary value problems for the linear and nonlinear time-fractional diffusion equations (for example \cite{Caff, Kir1, Nieto, Dong, Yam1, Luch1, Yam2, Luch2, Zacher, Zacher1, Zacher2, Zacher3} and references therein).

In addition to the well-posedness of problems for the evolution equations, the asymptotic behavior of solutions is of particular interest. The asymptotic estimates for the decay of solutions for evolution equations with time-fractional derivatives were obtained in a series of papers (for example \cite{Vald1, Vald2, Luch1, Yam2, Zacher1} and references therein).

For example,  Vergara and Zacher in \cite{Zacher1} studied the decay estimates of solutions to the equations
\begin{equation*}\partial^{\alpha}_{0+,t} u(t, x) - \Delta u(t, x) = 0,\qquad\qquad\qquad\text{Sub-diffusion equation},\end{equation*}
\begin{equation*}\partial^{\alpha}_{0+,t} u(t, x) - \Delta_p u(t, x) = 0,\qquad\text{p-Laplace sub-diffusion equation}\end{equation*} and
\begin{equation*}\partial^{\alpha}_{0+,t} u(t, x) - \Delta u^m(t, x) = 0,\qquad\text{Fractional porous medium equation}.\end{equation*}
For the above equations, the following decay estimates were obtained:

$\bullet$ Solution of sub-diffusion equation: $$\|u(t, \cdot)\|_{L^2(\Omega)}\leq\frac{C}{1+t^{\alpha}},\,t\geq 0,\,C>0, \alpha\in(0,1);$$

$\bullet$ Solution of p-Laplace sub-diffusion equation: $$\|u(t, \cdot)\|_{L^2(\Omega)}\leq\frac{C}{1+t^{\frac{\alpha}{p-1}}},\,t\geq 0,\,C>0,\, \alpha\in(0,1)\qquad\textit{if}\,\,\,\,\,\frac{2n}{n+2}\leq p<\infty,$$
$$\|u(t, \cdot)\|_{L^{\frac{n(2-p)}{p}}(\Omega)}\leq\frac{C}{1+t^{\frac{\alpha}{p-1}}},\,t\geq 0,\,C>0,\, \alpha\in(0,1)\qquad\textit{if}\,\,\,\,\,1<p<\frac{2n}{n+2};$$

$\bullet$ Solution of fractional porous medium equation: $$\|u(t, \cdot)\|_{L^{m+1}(\Omega)}\leq\frac{C}{1+t^{\frac{\alpha}{m}}},\,t\geq 0,\,C>0,\,\alpha\in(0,1)\qquad\textit{if}\,\,\,\,\,\frac{n-2}{n+2}\leq m<\infty,$$
$$\|u(t, \cdot)\|_{L^{\frac{n(1-m)}{2}}(\Omega)}\leq\frac{C}{1+t^{\frac{\alpha}{m}}},\,t\geq 0,\,C>0,\,\alpha\in(0,1)\qquad\textit{if}\,\,\,\,\,1<m<\frac{n-2}{n+2}.$$
In \cite{Vald2}, Dipierro, Valdinoci and Vespri considered a more general sub-diffusion equation
\begin{equation*}\partial^{\alpha}_{0+,t} u(t, x) + \mathcal{N}[u](t, x) = 0,\end{equation*}
where $\mathcal{N}$ one of the following operators:

$\bullet$ the Laplacian,

$\bullet$ the p-Laplacian,

$\bullet$ the porous medium operator,

$\bullet$ the doubly nonlinear operator,

$\bullet$ the mean curvature operator,

$\bullet$ the fractional Laplacian,

$\bullet$ the fractional p-Laplacian,

$\bullet$ the sum of different space-fractional operators,

$\bullet$ the fractional porous medium operator,

$\bullet$ the fractional mean curvature operator.

The main structural assumption they made was that there are $s \in (1, +\infty),$ $\gamma\in (0, +\infty),$ and $C>0$ such that if $u$ is a solution to the above sub-diffusion equation, then
$$\|u(t, \cdot)\|^{s-1+\gamma}_{L^s(\Omega)}\leq C\int\limits_\Omega u^{s-1}(t, x)\mathcal{N}[u](t, x)dx.$$
Under the above structural condition they obtain that
$$\|u(t, \cdot)\|_{L^s(\Omega)}\leq \frac{M}{1+t^{\frac{\alpha}{\gamma}}},\,M>0,\,\alpha\in(0,1).$$
\subsection{Motivation}
Consider the following toy model of the influence of the time-dependent coefficient on the diffusion process (for example see \cite{Bol, Cos, FaLe, Hris, Prus})
\begin{equation}\label{00}
\partial^{\alpha}_{0+,t} u(t, x) - t^\beta u_{xx}(t, x) = 0.
\end{equation}
This equation can be used to model sub-diffusive and super-diffusive processes in turbulent media. The mean squared displacement corresponding to this equation is given by
$$\langle x^2\rangle\sim t^{\alpha+\beta}.$$
As observed, the mean squared displacement demonstrates power law behavior. The sub-diffusive processes arise from $0<\alpha+\beta<1,$ while the super-diffusive processes arise from $\alpha+\beta>1.$ The normal diffusive processes arise from $\alpha+\beta=1,$ which is absent in the case of the time-fractional diffusion equations with constant coefficients, except for $\alpha=1$. Moreover, the complete stationary process can be formed for any time when $\alpha=-\beta$.

The practical significance of time-fractional diffusion equations with time-dependent coefficients lies in their ability to encapsulate physical phenomena. They serve as mathematical representations for both sub-diffusive and super-diffusive processes occurring in turbulent media. For instance, the modeling of pollutant transport in porous media can be effectively handled through the lens of sub-diffusive processes. On the other hand, the kinetics of particle motion in turbulent fluids can be understood more comprehensively using super-diffusive processes as a model.

To our understanding, existing literature on decay estimates for time-fractional evolution equations primarily focuses on cases involving constant coefficients, time-independent coefficients, or time-dependent coefficients within the $L^\infty$ (bounded) range. This observation fuels our interest in exploring the decay estimates for solutions of time-fractional evolution equations that feature time-dependent weakly singular coefficients.

\section{Preliminaries}
\subsection{Kilbas-Saigo function and its properties}
The Kilbas-Saigo function was introduced by Kilbas and Saigo \cite[Remark 5.1]{KS95} in terms of a special function of the form
\begin{equation}
\label{SF-01}
E_{\alpha, m, n}(z)=1+\sum_{k=1}^{\infty}\prod_{j=0}^{k-1}\frac{\Gamma(\alpha(jm+n)+1)}{\Gamma(\alpha(jm+n+1)+1)}\,z^{k},
\end{equation}
where $\alpha,\, m$ are real numbers and $n\in \mathbb{C}$ such that
\begin{equation}\label{cond1}
\alpha>0,\, m>0,\, \alpha(jm+n)+1\neq -1,-2,-3,... (j\in\mathbb N_0).
\end{equation}
In particular, if $m = 1,$ the function $E_{\alpha, m, n}(z)$ is reduced to the two-parameter Mittag-Leffler function:
\begin{equation*}
 E_{\alpha, 1, n}(z)=\Gamma(\alpha n+1)E_{\alpha, \alpha n+1}(z),
\end{equation*} where two-parameter Mittag-Leffler function $E_{\alpha,\beta}(z)$ is defined by
\begin{equation*}
E_{\alpha,\beta}(z)=\sum\limits_{k=0}^\infty \frac{z^k}{\Gamma(\alpha k+\beta)},\,\,\alpha>0,\, \beta>0,\, z\in \mathbb{C}.
\end{equation*} Two-parameter Mittag-Leffler function, sometimes called a Mittag-Leffler-type function, first appeared in \cite{W05}.
If $m = 1, n=0,$ then it coincides with the classical Mittag-Leffler function:
\begin{equation*}
 E_{\alpha, 1, 0}(z)=E_{\alpha, 1}(z),
\end{equation*} where a classical Mittag-Leffler function $E_{\alpha,1}(z)$ is defined by (\cite{M-L03})
\begin{equation*}
E_{\alpha,1}(z)=\sum\limits_{k=0}^\infty \frac{z^k}{\Gamma(\alpha k+1)},\,\,\alpha>0,\, z\in \mathbb{C},
\end{equation*} with the property $E_{1,1}(z)=e^z.$

Below we calculate a special case of the Kilbas-Saigo function when $\alpha=1,\,m>1,\,n=m-1$
\begin{align*}
E_{1, m, m-1}(z)&=1+\sum_{k=1}^{\infty}\prod_{j=0}^{k-1}\frac{\Gamma((jm+m-1)+1)}{\Gamma((jm+m)+1)}\,z^{k}
\\&=1+\sum_{k=1}^{\infty}\prod_{j=0}^{k-1}\frac{\Gamma(jm+m)}{\Gamma(jm+m+1)}\,z^{k}\\& =1+\sum_{k=1}^{\infty}\prod_{j=0}^{k-1}\frac{1}{(jm+m)}\,z^{k} \\& =1+\sum_{k=1}^{\infty}\frac{1}{k!}\left(\frac{z}{m}\right)^{k}=e^{\frac{z}{m}},\end{align*} which gives $E_{1, m, m-1}(z)=e^{\frac{z}{m}}.$

In the course to proving our main results for the linear cases, one of the key roles is played the following properties of the Kilbas-Saigo function obtained in \cite[Theorem 2]{TSim19}
\begin{equation}\label{estim1}
\frac{1}{1+\Gamma(1-\alpha)z}\leq E_{\alpha, m, m-1}(-z)\leq \frac{1}{1+\frac{\Gamma(1+(m-1)\alpha)}{\Gamma(1+m\alpha)}z},\, z\geq 0,
\end{equation} where $m>1$ and $0<\alpha<1$.
\subsection{Semi-linear fractional differential equation}
In this subsection, we study the semi-linear equation
\begin{equation}\label{ode}\partial^{\alpha}_{0+,t}H(t) + \nu t^{\beta}H^\delta(t) = 0,\,t>0,\,\nu>0,\,\delta>0.\end{equation}
\begin{lemma}\label{lem1}
Let $\alpha\in(0,1),$ $\beta>-\alpha,$ $\delta>0$ and $\nu > 0.$ Suppose that $H\in C(\mathbb{R}_+)$ is the solution of \eqref{ode} with the initial data $H(0)=H_0>0.$ Then, it holds
\begin{equation}\label{ode1}\frac{c_1}{1+t^{\frac{\alpha+\beta}{\delta}}}\leq H(t)\leq \frac{c_2}{1+t^{\frac{\alpha+\beta}{\delta}}},\,c_1,c_2>0,\,t\geq 0.\end{equation}
\end{lemma}
\begin{proof}
We will prove the decay estimate of the solution to equation \eqref{ode} using the upper and lower solutions method. Let us consider the functions
$$\hat{H}(t)=\left\{\begin{array}{l}H_0-\nu\Gamma(1-\alpha)H_0^{\delta}t^{\alpha+\beta},\,t\in\left[0, t_1\right],\\{}\\ \left(\frac{H_0^{1-\delta}}{2\nu\Gamma(1-\alpha)}\right)^{\frac{1}{\delta}}2^{-1}H_0t^{-\frac{\alpha+\beta}{\delta}},\,t\geq t_1\end{array}\right.$$ and
$$\tilde{H}(t)=\left\{\begin{array}{l}H_0, \,t\in\left[0, t_2\right],\\{}\\ H_0t_2^{\frac{\alpha+\beta}{\delta}}t^{-\frac{\alpha+\beta}{\delta}},\,t\geq t_2,\end{array}\right.$$ where $$t_1=\left(\frac{H_0^{1-\delta}}{2\nu\Gamma(1-\alpha)}\right)^{\frac{1}{\alpha+\beta}}$$ and $$t_2=\left(\frac{H_0^{1-\delta}}{\nu}\bigg(\frac{2^{\alpha}}{\Gamma(1-\alpha)}+\frac{\alpha+\beta}{\delta} \frac{2^{\alpha+\frac{\alpha+\beta}{\delta}}}{\Gamma(2-\alpha)}\bigg)\right)^{\frac{1}{\alpha+\beta}}.$$
Having carried out calculations similar to \cite[Section 7]{Zacher1}, it is easy to prove the function $\tilde{H}(t)$ will be a super-solution, and $\hat{H}(t)$ will be a sub-solution of equation \eqref{ode} with the initial data $H(0)=H_0$.
It means, there exist positive constants $c_1>0$ and $c_2>0$ such that
\begin{equation*}\frac{c_1}{1+t^{\frac{\alpha+\beta}{\delta}}}\leq H(t)\leq \frac{c_2}{1+t^{\frac{\alpha+\beta}{\delta}}},\,t\geq 0.\end{equation*}
The proof is complete.
\end{proof}

\section{Decay rates for the linear evolution equations}
\subsection{Linear anomalous diffusion equations}
Let $a(t)=t^\beta,\,\beta>-\alpha$ and $\mathcal{A}(u):=\Delta u,$ then the equation \eqref{1.1} coincide with the sub-diffusion equation
\begin{equation}\label{1.1-1}\partial^{\alpha}_{0+,t} u(t, x) - t^{\beta}\Delta u(t, x) = 0,\,\left({t,x}
\right) \in \mathbb{R}_+\times\Omega:=\Omega_+.\end{equation}

At this thanks to the elements of the theories of harmonic analysis, in this stage, we are able to study in detail the properties of the solution of problem \eqref{1.1-1}, \eqref{1.2}, \eqref{1.3}.

Let $H^2(\Omega)$ is a Hilbert space defined by
$$H^2(\Omega)=\{f\in L^2(\Omega):\, \sum\limits_{k = 1}^\infty \lambda^2_k|(f,e_k)|^2<\infty\},$$
with the norm
$$\|f\|^2_{H^2(\Omega)}=\sum\limits_{k = 1}^\infty \lambda^2_k|(f,e_k)|^2,$$
where $\{\lambda_k> 0,\;k\in\mathbb N\}$ are eigenvalues and $\{e_k\in L^2(\Omega), k\in\mathbb N\}$ are corresponding orthonormal eigenfunctions of Dirichlet-Laplacian problem:
\begin{equation}\label{DirLap}\left\{\begin{array}{l}-\Delta e_k(x)=\lambda_k e_k(x),\,x\in\Omega,\\{}\\ e_k(x)=0,\,\partial\Omega.\end{array}\right.
\end{equation}

\begin{definition} The generalized solution of problem \eqref{1.1-1}, \eqref{1.2}, \eqref{1.3} is a bounded function $u\in C\left([0,+\infty);L^2(\Omega)\right),$ such that $t^{-\beta}\partial_{0+,t}^{\alpha } u, \Delta u\in C\left((0,+\infty);L^2(\Omega)\right).$\end{definition}
\begin{theorem}\label{th1}Let $\alpha\in(0,1),\,\beta>-\alpha$ and $u_0 \in H^2(\Omega).$ Then the generalized solution $u$ of problem \eqref{1.1-1}, \eqref{1.2}, \eqref{1.3} exists, it is unique and can be represented as \begin{equation}\label{1.5}u\left( {t, x} \right) = \sum\limits_{k = 1}^\infty {u_{0k} E_{\alpha, 1+\frac{\beta}{\alpha}, \frac{\beta}{\alpha}} \left( {- \lambda_k t^{\alpha+\beta}  }\right) e_k\left( x \right)},\,(t, x)\in \Omega_+,\end{equation} where $u_{0k}=\int\limits_\Omega {u_0\left( x \right) {e_k \left( x \right)}dx},\,k\in\mathbb{N},$ and $E_{\alpha, m, l} \left( z \right)$ is the Kilbas-Saigo function.

In addition, the solution $u$ satisfies the following decay rates:
\begin{align}\label{DEst}\frac{m_1}{1+\lambda_1t^{\alpha+\beta}}\leq\|u(t, \cdot)\|_{L^2(\Omega)}\leq\frac{M_1}{1+\lambda_1t^{\alpha+\beta}},\,t\geq 0,\end{align} where $m_1$ and $M_1$ are positive constants.
\end{theorem}
\begin{remark} The results of Theorem \ref{th1} remain valid for more general positive self-adjoint operators with discrete spectrum, such as Robin-Laplacian, fractional Laplacian, poly-Laplacian, Sturm-Liouville operator with involution, fractional Sturm-Liouville operator and others. In these cases, the proofs are carried out by simply repeating the proof of Theorem \ref{th1}.
But in the case of the \textbf{Neumann} boundary conditions, it will not be possible to obtain the decay rates of solutions. The solution of this case can be written in the form
\begin{equation}\label{Neumann}u\left( {t, x} \right) = \sum\limits_{k = 0}^\infty {u_{0k} E_{\alpha, 1+\frac{\beta}{\alpha}, \frac{\beta}{\alpha}} \left( {- \lambda_k t^{\alpha+\beta}  }\right) e_k\left( x \right)},\,(t, x)\in \Omega_+\end{equation}
and it does not decrease for large $t\rightarrow\infty$, but remains bounded for all $t\geq 0$. Using the technique of the proof of Theorem \ref{th1} (given below), it is not difficult to show that
\begin{align*}\sqrt{|u_{00}|}\leq\|u(t, \cdot)\|_{L^2(\Omega)}\leq C_1\|u_0\|_{L^2(\Omega)},\,\,t\geq 0,\end{align*}
where $u_{00}=\int\limits_\Omega u_0(x)dx$ and $0<C_1$ is the arbitrary constant.

This is due to the fact that the first eigenvalue of the Neumann-Laplace operator is equal to zero. In the general case, one can obtain the following estimate
\begin{equation}
\label{DEstN}\begin{split}\sqrt{|u_{00}|}+\frac{c\sqrt{|u_{01}|}}{1+\lambda^N_2t^{\alpha+\beta}}&\leq\|u(t, \cdot)\|_{L^2(\Omega)}\\&\leq \sqrt{|u_{00}|}+\frac{C\|u_0\|_{{L^2(\Omega)}}}{1+\lambda^N_2t^{\alpha+\beta}},\,t\geq 0,\end{split}\end{equation} where $C, c>0$ are arbitrary constants, $u_{01}=\int\limits_\Omega u_0(x)e_2(x)dx,$ $e_2(x)$ is the second eigenfunction and $\lambda_2^N>0$ is the corresponding second eigenvalue of Neumann-Laplacian problem:
\begin{equation*}\left\{\begin{array}{l}-\Delta e_k(x)=\lambda_k e_k(x),\,x\in\Omega,\\{}\\ \partial_\nu e_k(x)=0,\,\partial\Omega.\end{array}\right.
\end{equation*} Here $\partial_\nu$ is the outside normal derivative.

In \eqref{DEstN}, to obtain the decay estimate, it is necessary and sufficient that the condition $\int\limits_\Omega {u_0\left( x \right)dx}=0$ be satisfied for the initial data $u_0.$ Then the solution of the Cauchy-Neumann problem satisfies the estimate
\begin{equation*}
\label{DEst+}\begin{split}\frac{c_1}{1+\lambda^N_2t^{\alpha+\beta}}\leq\|u(t, \cdot)\|_{L^2(\Omega)}\leq \frac{C_1}{1+\lambda^N_2t^{\alpha+\beta}},\,t\geq 0,\end{split}\end{equation*}
that is, the decay rate will be sharp up to the constant.

Similar reasoning is also valid for the fractional Laplace with the nonlocal Neumann condition
\begin{equation*}\left\{\begin{array}{l}(-\Delta)^s e_k(x)=\lambda_k e_k(x),\,x\in\Omega,\\{}\\ \mathcal{N}_s e_k(x)=0,\,\mathbb{R}^n\setminus\bar\Omega,\end{array}\right.
\end{equation*} where $\mathcal{N}_s$ is a nonlocal normal derivative, given by
$$\mathcal{N}_s e_k=c_{n,s}\int\limits_\Omega\frac{e_k(x)-e_k(y)}{|x-y|^{n+2s}}dy,$$
where $c_{n,s}$ is a normalization constant.

We have to note that in \cite{Vald4} it was proved that the above nonlocal Neumann-Laplace eigenvalue problem has eigenfunctions forming an orthogonal basis in $L^2(\Omega)$, and the corresponding eigenvalues are nonnegative:
$$0=\lambda_1<\lambda_2\leq\lambda_3\leq ....$$
\end{remark}
\begin{proof}[Proof of Theorem \ref{th1}] We will carry out the proof step by step.\\
{\bf Existence of solution.}
The solution of problem \eqref{1.1-1}, \eqref{1.2}-\eqref{1.3} can be represented as: \begin{equation}\label{3.2}u\left( {t,x} \right) = \sum\limits_{k = 1}^\infty {u_k
\left( t \right) e_k\left( x \right)},\,\,\,\text{in}\,\, \Omega_+.\end{equation} It is clear that if $u_0\in H^2(\Omega)$, then
it follows
\begin{equation*}u_0\left( x \right) = \sum\limits_{k = 1}^\infty {u_{0k}
e_k \left( x \right)},\,x\in\Omega,\end{equation*} where
$u_{0k}=\int\limits_\Omega {u_0\left( x \right) {e_k \left( x \right)}dx}.$

Substituting function \eqref{3.2} into equation \eqref{1.1}, we obtain the following problem for
$u_k(t),$
\begin{equation}\label{3.3}\partial_{0+,t}^{\alpha} u_k
\left( t \right) + \lambda_k t^{\beta}u_k \left( t \right) = 0,\, t >0,\end{equation} \begin{equation}\label{3.4}u_k \left(0\right) = u_{0k}.\end{equation}
Its known \cite[P. 233]{1} that the unique solution to problem \eqref{3.3}-\eqref{3.4} has the form \begin{equation}\label{3.6}u_k \left(
t \right) = u_{0k} E_{\alpha, 1+\frac{\beta}{\alpha}, \frac{\beta}{\alpha}} \left( {- \lambda_k t^{\alpha+\beta}  }\right),\end{equation}
hence
\begin{equation*}u\left( {t,x} \right) = \sum\limits_{k = 1}^\infty {u_{0k} E_{\alpha, 1+\frac{\beta}{\alpha}, \frac{\beta}{\alpha}} \left( {- \lambda_k t^{\alpha+\beta}  }\right) e_k\left( x \right)}.\end{equation*}
{\bf Convergence of solution.}
The estimate \eqref{estim1} gives
\begin{align*}|u_k \left(
t \right)| \leq \frac{|u_{0k}|}{1+\frac{\Gamma(\beta+1)}{\Gamma(\alpha+\beta+1)}\sqrt{\lambda_k} t^{\alpha+\beta}}\leq |u_{0k}|,\end{align*}
which implies
\begin{align*}\|u\left( {t,\cdot} \right)\|_{L^2(\Omega)}^2 &\leq \sum\limits_{k = 1}^\infty {|u_{0k}|^2 \left|E_{\alpha, 1+\frac{\beta}{\alpha}, \frac{\beta}{\alpha}} \left( {- \lambda_k t^{\alpha+\beta}  }\right)\right|^2 \|e_k\|_{L^2(\Omega)}^2}\\& \leq \sum\limits_{k = 1}^\infty \frac{|u_{0k}|^2}{\left(1+\frac{\Gamma(\beta+1)}{\Gamma(\alpha+\beta+1)}\lambda_k t^{\alpha+\beta}\right)^2} \\&\leq \sum\limits_{k = 1}^\infty {|u_{0k}|^2}=\|u_{0}\|_{L^2(\Omega)}^2<\infty,\end{align*} thanks to Parseval's identity.

Let us calculate $\partial_{0+,t}^{\alpha } u$ and $\Delta u.$ We have
\begin{align*}\partial^\alpha_{0+,t}u\left( {t,x} \right)& = \sum\limits_{k = 1}^\infty {u_{0k} \left(\partial^\alpha_{0+,t}E_{\alpha, 1+\frac{\beta}{\alpha}, \frac{\beta}{\alpha}} \left( {- \lambda_k t^{\alpha+\beta}  }\right)\right) e_k\left( x \right)}\\&= -t^{\beta}\sum\limits_{k = 1}^\infty \lambda_k{u_{0k} E_{\alpha, 1+\frac{\beta}{\alpha}, \frac{\beta}{\alpha}} \left( {- \lambda_k t^{\alpha+\beta}  }\right) e_k\left( x \right)},\,(t,x)\in \Omega_+\end{align*} and \begin{align*}\Delta u\left( {t,x} \right)& = \sum\limits_{k = 1}^\infty {u_{0k} E_{\alpha, 1+\frac{\beta}{\alpha}, \frac{\beta}{\alpha}} \left( {- \lambda_k t^{\alpha+\beta}  }\right) \Delta e_k\left( x \right)}\\&= -\sum\limits_{k = 1}^\infty \lambda_k{u_{0k} E_{\alpha, 1+\frac{\beta}{\alpha}, \frac{\beta}{\alpha}} \left( {- \lambda_k t^{\alpha+\beta}  }\right) e_k\left( x \right)},\,(t,x)\in \Omega_+.\end{align*}
Applying the above calculations and Parseval's identity we deduce that
\begin{align*}\left\|t^{-\beta}\partial_{0+,t}^{\alpha}u\left( {t,\cdot} \right)\right\|^2_{L^2(\Omega)} \leq \sum\limits_{k = 1}^\infty \lambda_k^2{|u_{0k}|^2 }=\|u_{0}\|^2_{H^2(\Omega)}<\infty \end{align*} and \begin{align*}\|\Delta u\left({t,\cdot} \right)\|^2_{L^2(\Omega)}\leq\sum\limits_{k = 1}^\infty \lambda_k^2{|u_{0k}|^2}=\|u_{0}\|^2_{H^2(\Omega)}<\infty.\end{align*}
{\bf Uniqueness of solution.} Suppose that there are two solutions $u_1(t,x)$ and $u_2(t,x)$ of problem \eqref{1.1-1}, \eqref{1.2}--\eqref{1.3}. Let $u(t,x)=u_1(t,x)-u_2(t,x).$ Then $u(t,x)$ satisfies the equation \eqref{1.1-1} and the homogeneous conditions \eqref{1.2}--\eqref{1.3}.

Let us consider the function \begin{equation}\label{3.5}u_k(t)=\int\limits_\Omega u(t,x){e_k(t)}dx,\,k\in\mathbb{N},\,t\geq 0.\end{equation}
Applying $\partial^{\alpha}_{0+,t}$ to the function \eqref{3.5} by \eqref{1.1} we have
\begin{align*}\partial^{\alpha}_{0+,t}u_k(t)&=\int\limits_\Omega \partial^{\alpha}_{0+,t}u(t,x){e_k(x)}dx =t^{\beta}\int\limits_\Omega \Delta u(t,x){e_k(x)}dx \\& =t^{\beta}\int\limits_\Omega u(t,x)\Delta{e_k(x)}dx=t^{\beta}\lambda_k\int\limits_\Omega u(t,x){e_k(x)}dx\\&= t^{\beta}\lambda_k u_k(t),\,k\in\mathbb{N},\,t\geq 0.\end{align*} Also from \eqref{1.2} and \eqref{1.3} we have
$u_k(0)=0.$ Then from \eqref{3.6} we conclude that $$u_k(t)=0,\,t\geq 0.$$ This implies $$\int\limits_\Omega u(t,x){e_k(x)}dx=0,$$ and the completeness of the system $e_k(x),\,k\in \mathbb{N},$ gives $u(t,x)\equiv 0,\,(t,x)\in \Omega_+.$\\
{\bf Optimal decay of solution.}
First we will prove the upper bound of $u$. It is known that the eigenvalues $\lambda_k$ of Dirichlet-Laplacian do not increase, that is $0<\lambda_1\leq\lambda_2\leq...\leq \lambda_k\leq...\nearrow+\infty.$ Then for the solution $u$ we have the following estimate
\begin{align*}\|u\left( {t,\cdot} \right)\|_{L^2(\Omega)}^2 &\leq \sum\limits_{k = 1}^\infty {|u_{0k}|^2 \left|E_{\alpha, 1+\frac{\beta}{\alpha}, \frac{\beta}{\alpha}} \left( {- \lambda_k t^{\alpha+\beta}  }\right)\right|^2 \|e_k\|_{L^2(\Omega)}^2}\\& \leq \sum\limits_{k = 1}^\infty \frac{|u_{0k}|^2}{\left(1+\frac{\Gamma(\beta+1)}{\Gamma(\alpha+\beta+1)}\lambda_k t^{\alpha+\beta}\right)^2} \\& \leq \frac{1}{\left(1+\frac{\Gamma(\beta+1)}{\Gamma(\alpha+\beta+1)}\lambda_1 t^{\alpha+\beta}\right)^2}\sum\limits_{k = 1}^\infty |u_{0k}|^2 \\&\leq \frac{M_1^2}{\left(1+ \lambda_1 t^{\alpha+\beta}\right)^2},\end{align*} hence
\begin{align*}\|u\left( {t,\cdot} \right)\|_{L^2(\Omega)} \leq \frac{M_1}{1+ \lambda_1 t^{\alpha+\beta}}.\end{align*}
Now we will prove the lower bound of $u$. Using Parseval's identity, we obtain
\begin{align*}\|u\left( {t,\cdot} \right)\|_{L^2(\Omega)}^2 &=\sum\limits_{k = 1}^\infty {|u_{0k}|^2 \left|E_{\alpha, 1+\frac{\beta}{\alpha}, \frac{\beta}{\alpha}} \left( {- \lambda_k t^{\alpha+\beta}  }\right)\right|^2 \|e_k\|_{L^2(\Omega)}^2}\\& \geq \sum\limits_{k = 1}^\infty \frac{|u_{0k}|^2}{\left(1+\Gamma(1-\alpha)\lambda_k t^{\alpha+\beta}\right)^2} \\&= \frac{|u_{01}|^2}{\left(1+\Gamma(1-\alpha)\lambda_1 t^{\alpha+\beta}\right)^2}+\sum\limits_{k = 2}^\infty \frac{|u_{0k}|^2}{\left(1+\Gamma(1-\alpha)\lambda_k t^{\alpha+\beta}\right)^2}\\&\geq \frac{m_1^2}{\left(1+ \lambda_1t^{\alpha+\beta}\right)^2}.\end{align*} The proof is complete.
\end{proof}
\subsection{Sub-diffusion equation with general diffusive coefficient}
In this subsection we consider a general case of sub-diffusion equation \eqref{1.1-1}, i.e.
\begin{equation}\label{1.1-1*}\partial^{\alpha}_{0+,t}u(t,x) - a(t)\Delta u(t,x) = 0,\,\left({t,x}
\right) \in \mathbb{R}_+\times\Omega:=\Omega_+,\end{equation} where $a(t)$ is a nonnegative function.

We recall that the existence and uniqueness of a weak local solution $u$ to problem \eqref{1.1-1*}, \eqref{1.2}-\eqref{1.3} were obtained in \cite{Zacher3}. We will be interested in the behavior of the solution for large $t.$
\begin{theorem}\label{th1+} Let $u$ is a solution of problem \eqref{1.1-1*}, \eqref{1.2}, \eqref{1.3}. Suppose that  $u_0 \in L^2(\Omega)$ and $a(t)$ satisfies (H), then the solution $u$ of problem \eqref{1.1-1*}, \eqref{1.2}, \eqref{1.3} satisfies the estimate:
\begin{align}\label{DEst2}\|u(\cdot,t)\|_{L^2(\Omega)}\leq\frac{M}{1+\lambda_1\kappa t^{\alpha+\beta}},\,t\geq 0,\end{align} where $M$ is a positive constant.
\end{theorem}
\begin{remark} Note that, the estimate \eqref{DEst2} is optimal, if the function $a(t)$ is bounded from above as follows
$$0<a(t)\leq \kappa_1 t^\beta,\,\,\kappa_1>0,\,t>0.$$
Indeed, let us consider the function
$$v(t,x)=T(t)e_1(x),$$ where $e_1(x)$ is a first eigenfunction of Dirichlet-Laplacian problem \eqref{DirLap}.

Substituting the function $v$ to equation \eqref{1.1-1*} one can get
\begin{align*}0&=\partial^{\alpha}_{0+,t}T(t) + \lambda_1a(t)T(t)\\& \leq \partial^{\alpha}_{0+,t}T(t) + \lambda_1\kappa_1 t^\beta T(t),\end{align*}
which implies that $v$ is a sub-solution of the equation \eqref{1.1-1*}.
\end{remark}
\begin{proof}[Proof of Theorem \ref{th1+}] Multiplying each term in \eqref{1.1-1*} by $u$, and then integrating over $\Omega$, we obtain
\begin{equation*}\int\limits_\Omega u[\partial^{\alpha}_{0+,t}u]dx - a(t)\int\limits_\Omega u\Delta udx = 0.\end{equation*}
Using the inequality (see \cite{Zacher1}) \begin{equation}\label{Lp}
\|u(t,\cdot)\|_{L^2(\Omega)}\partial^{\alpha}_{0+,t}\left(\|u(t,\cdot)\|_{L^2(\Omega)}\right)\leq\int\limits_{\Omega}u[\partial^{\alpha}_{0+,t}u]dx,
\end{equation} and integrating by parts the last in equality, we arrive at
\begin{equation*}\|u(t,\cdot)\|_{L^2(\Omega)}\partial^{\alpha}_{0+,t}\left(\|u(t,\cdot)\|_{L^2(\Omega)}\right) + a(t)\int\limits_\Omega |\nabla u|^2dx\leq 0.\end{equation*}
Now, taking into account (H) and applying the Poincar\'{e} inequality, we obtain
\begin{equation*}\partial^{\alpha}_{0+,t}\left(\|u(t,\cdot)\|_{L^2(\Omega)}\right) + \lambda_1\kappa t^\beta\|u(t,\cdot)\|_{L^2(\Omega)}\leq 0.\end{equation*}
This means that the solution of \eqref{1.1-1*} is a super solution of equation \eqref{1.1-1}, which completes the proof.
\end{proof}
\begin{remark}In a similar way, one can consider a more general non-linear operator as $$\mathcal{A}(u):=\textrm{div}(A(t,x,u,\nabla u)),$$ where $A$ satisfies
$$(A(t,x,u,v),v)\geq \kappa (v,v),\,\kappa>0.$$
Then, repeating the same technique above, one can show that the solution of equation
\begin{equation*}\partial^{\alpha}_{0+,t} u(t,x) - a(t)\textrm{div}(A(t,x,u,\nabla u)) = 0,\,\left({t,x}
\right) \in \Omega_+,\end{equation*}
satisfies the estimate
\begin{align*}\|u(t,\cdot)\|_{L^2(\Omega)}\leq\frac{C}{1+ t^{\alpha+\beta}},\,t\geq 0,\end{align*} where $C$ is a positive constant.
\end{remark}
\subsection{The case of heat equation with general diffusive coefficient}
In this subsection we consider the following heat equation
\begin{equation}\label{1.1-2}u_t(t,x) - a(t)\Delta u(t,x) = 0,\,\left({t,x}
\right) \in\mathbb{R}_+\times \Omega:=\Omega_+,\end{equation} where $a(t)$ is a given function satisfies the condition $$\int\limits_0^t a(s)ds>0,\,t>0.$$
In contrast to other sections, here we can consider a more general function $a(t)$ without restricting condition (H).

\begin{theorem}\label{th2} Let $\int\limits_0^t a(s)ds>0,\,t>0,$ and $u_0 \in H^2(\Omega).$ Then the generalized solution $u$ of problem \eqref{1.1-2}, \eqref{1.2}, \eqref{1.3} exists, it is unique and given as \begin{equation}\label{1.5}u\left( {t,x} \right) = \sum\limits_{k = 1}^\infty {u_{0k} e^{- \lambda_k \int\limits_0^t a(s)ds} e_k\left( x \right)},\,(t,x)\in \Omega_+,\end{equation}
furthemore, the solution $u$ satisfies the following decay rates:
\begin{align*}m_1e^{- \lambda_1 \int\limits_0^t a(s)ds}\leq\|u(t,\cdot)\|_{L^2(\Omega)}\leq M_1e^{- \lambda_1 \int\limits_0^t a(s)ds},\,t\geq 0,\end{align*} where $m_1$ and $M_1$ are positive constants.
\end{theorem}
\begin{proof} The proof of Theorem \ref{th2} is identical with the proof of Theorem \ref{th1}. The main difference is that in Theorem \ref{th2} the key role is played by the differential equation
$$u'_k(t)+\lambda_ka(t)u_k(t)=0,$$
instead of equation \eqref{3.3}.
\end{proof}
Choosing the function $a(t)$ explicitly, one can obtain different types of decay of the solutions. Below we present some special cases of the function $a(t).$
\begin{description}
  \item[Exponential decay] Assume that $$a(t):=\beta t^{\beta-1},\,\beta>0,$$ then the decay rate has the form
\begin{align*}m_1e^{- \lambda_1 t^{\beta}}\leq\|u(t,\cdot)\|_{L^2(\Omega)}\leq M_1e^{- \lambda_1 t^{\beta}},\,t\geq 0\end{align*} and the solution can be represented as
\begin{equation*}u\left( {t,x} \right) = \sum\limits_{k = 1}^\infty {u_{0k} e^{- \lambda_1 t^{\beta}} e_k\left( x \right)},\,(t,x)\in \Omega_+.\end{equation*}
  \item[Logarithmic decay] Let $$a(t):=\frac{p}{(1+\log(1+t))(1+t)},\,p>0,$$ then it holds
\begin{align*}\frac{m_1}{(1+\log(1+t))^{p\lambda_1}}\leq\|u(t,\cdot)\|_{L^2(\Omega)}\leq \frac{M_1}{(1+\log(1+t))^{p\lambda_1}},\,t\geq 0.\end{align*} In this case the solution given by
\begin{equation*}u\left( {t,x} \right) = \sum\limits_{k = 1}^\infty {u_{0k} (1+\log(1+t))^{- \lambda_k p} e_k\left( x \right)},\,(t,x)\in \Omega_+.\end{equation*}
  \item[Polynomial decay] Suppose that $$a(t):=q\frac{\sum\limits_{j=1}^mja_jt^{j-1}}{\sum\limits_{j=0}^ma_jt^j},\,q>0,\,a_j>0,\,j=1,...,m,$$ hence, it yields
\begin{align*}m_1\left(\sum\limits_{j=0}^ma_jt^j\right)^{- q\lambda_1}\leq\|u(t,\cdot)\|_{L^2(\Omega)}\leq M_1\left(\sum\limits_{j=0}^ma_jt^j\right)^{- q\lambda_1},\,t\geq 0,\end{align*} and the solution
\begin{equation*}u\left( {t,x} \right) = \sum\limits_{k = 1}^\infty {u_{0k} \left(\sum\limits_{j=0}^ma_jt^j\right)^{- q\lambda_k} e_k\left( x \right)},\,(t,x)\in \Omega_+.\end{equation*}
\end{description}

\section{Decay estimates to the time-fractional nonlinear evolution equations}
In this section, we will study decay estimates of the solutions of some time-fractional nonlinear evolution equations.

For the proof of our result in this section, we will make essential use of the following inequality proved in \cite{Zacher1}
\begin{equation}\label{Lp}
\|u(t,\cdot)\|_{L^2(\Omega)}\partial^{\alpha}_{0+,t}\left(\|u(t,\cdot)\|_{L^2(\Omega)}\right)\leq\int\limits_{\Omega}u[\partial^{\alpha}_{0+,t}u]dx.
\end{equation}
\subsection{$p$-Laplacian sub-diffusion equation}
In this subsection, we focus on the case $$\mathcal{A}(u):=\Delta_p u=div\left(|\nabla u|^{p-2}\nabla u\right),$$ where we study the $p$-Laplacian sub-diffusion equation
\begin{equation}\label{1.1-3}\partial^{\alpha}_{0+,t} u(t,x) - a(t)\Delta_p u(t,x) = 0,\,\left({t,x}
\right) \in \Omega_+.\end{equation}

In general, the existence and uniqueness of a weak local solution $u$ to problem \eqref{1.1-3}, \eqref{1.2}-\eqref{1.3} were obtained in \cite{Zacher2} using the Galerkin approximation method. We will study the behavior of the solution for large $t.$
\begin{theorem}\label{th3}
Let $u$ be a solution of problem \eqref{1.1-3}, \eqref{1.2}-\eqref{1.3} and the function $a(t)$ satisfies (H).
Assume that $\frac{2n}{n+2}\leq p<\infty$ and let $u_0\in L^2(\Omega).$ Then
  \begin{equation}\label{p1}\|u(t,\cdot)\|_{L^2(\Omega)}\leq \frac{M}{1+t^{\frac{\alpha+\beta}{p-1}}},\,t\geq 0,
  \end{equation} where $M>0$ is an arbitrary constant depending on the $L^2(\Omega)$ norm of $u_0$.
\end{theorem}
\begin{remark}
Let $p=2$, then the estimate \eqref{p1} coincides with the estimate \eqref{DEst2} without regard to constants.
\end{remark}
\begin{remark}
Note that problem \eqref{1.1-3}, \eqref{1.2}-\eqref{1.3} was studied in \cite{Zacher1} in the case $\beta=0,$ and the following estimate was obtained
\begin{equation*}\|u(t,\cdot)\|_{L^2(\Omega)}\leq \frac{M}{1+t^{\frac{\alpha}{p-1}}},\,\,\,\text{for}\,\,\,\frac{2n}{n+2}\leq p<\infty,\,t\geq 0,
  \end{equation*}
which is a particular case of estimate \eqref{p1}.
\end{remark}
\begin{remark}
The decay rate \eqref{p1} is optimal for $a(t)=t^\beta$, at least for $p > \frac{2n}{n+2}.$ It can be shown in a similar way as in \cite[P. 236]{Zacher1}.
\end{remark}
\begin{proof}[Proof of Theorem \ref{th3}]
Multiplying \eqref{1.1-3} by $u$ and integrate over $\Omega$ one can obtain
\begin{equation*}\int_{\Omega}u[\partial^{\alpha}_{0+,t}u]dx+ a(t)\|\nabla u(t,\cdot)\|_{L^p(\Omega)}^p= 0,\,\,\, t>0.\end{equation*}
Suppose that $\frac{2N}{N+2}\leq p < \infty.$ Applying the inequality \eqref{Lp} with $q=2$ to the first term of the last equality and applying the Sobolev embedding $$H^1_p(\Omega)\hookrightarrow L_2(\Omega),\, \frac{2N}{N+2}\leq p < \infty,$$ to the second term we obtain
\begin{equation*}\|u(t,\cdot)\|_{L^2(\Omega)}\partial^{\alpha}_{0+,t}\left(\|u(t,\cdot)\|_{L^2(\Omega)}\right)+ Ca(t)\|u(t,\cdot)\|_{L^2(\Omega)}^p\leq 0,\,\,\, t>0.\end{equation*}
Hence, from the last inequality, taking into account (H), it follows that
\begin{equation*}\partial^{\alpha}_{0+,t}\|u(t,\cdot)\|_{L^2(\Omega)}+ Ct^{\beta}\|u(t,\cdot)\|_{L^2(\Omega)}^{p-1}\leq 0,\,\,\, t>0,\end{equation*}
where $C$ is a positive constant. Then $E(t)=\|u(t,\cdot)\|_{L^2(\Omega)}$ is a subsolution of the equation \eqref{ode}. Consequently $\|u(t,\cdot)\|_{L^2(\Omega)}\leq \frac{M}{1+t^{\frac{\alpha+\beta}{p-1}}}.$
The proof is complete.
\end{proof}

\subsection{Time-fractional porous medium equation}
In this subsection, we focus on $$\mathcal{A}(u):=\nabla(g(u) \nabla u),$$ where we study the time-fractional porous-medium equation
\begin{equation}\label{1.1-4}\partial^{\alpha}_{0+,t} u(t,x) - a(t)\nabla(g(u(t,x)) \nabla u(t,x)) = 0,\,\left({t,x}
\right) \in\mathbb{R}_+\times \Omega:=\Omega_+.\end{equation}
Suppose that the function $g(u)$ has the property:
\begin{description}
\item[(H1)] $g(u)\geq c_0 u^{m},\, m\geq 0,\,c_0>0.$
\end{description}

The existence of a weak local solution $u$ to problem \eqref{1.1-4}, \eqref{1.2}-\eqref{1.3} was proven in \cite{Zacher3} using the Galerkin approximation method. We intend to obtain the behavior of the solution for large $t.$
\begin{theorem}\label{th4}
Let $u$ be a solution of problem \eqref{1.1-4}, \eqref{1.2}-\eqref{1.3}. Let $a(t)$ and $g(u)$ satisfy the hypotheses (H) and (H1).
Assume that $u_0\in L^{2}(\Omega)$, then
  \begin{equation}\label{m1}\|u(t,\cdot)\|_{L^{2}(\Omega)}\leq \frac{M}{1+t^{\frac{\alpha+\beta}{m+1}}},\,t\geq 0,
  \end{equation} where $M>0$ is an arbitrary constant depending on the $L^2(\Omega)$ norm of $u_0$.
\end{theorem}
\begin{remark}
The problem \eqref{1.1-3}, \eqref{1.2}-\eqref{1.3} was studied in \cite{Vald2} when $\beta=0,$ and the following estimate was obtained
\begin{equation*}\|u(t,\cdot)\|_{L^2(\Omega)}\leq \frac{M}{1+t^{\frac{\alpha}{m+1}}},\,t\geq 0,
  \end{equation*}
which is a particular case of estimate \eqref{m1}.
\end{remark}
\begin{remark}
The decay rate \eqref{m1} is optimal for $a(t)=t^\beta$ and $g(u)=u^{m}$, at least for $m > \frac{2n}{n+2}.$ It can be proven in the same way as in \cite[P. 238]{Zacher1}.
\end{remark}
\begin{proof}[Proof of Theoren \ref{th4}] Multiplying the equation \eqref{1.1-4} by $u$ and integrating by parts over the domain $\Omega$ we get
$$\int_{\Omega} u(t,x) [\partial^{\alpha}_{0+,t} u(t,x)]dx+\kappa t^{\beta}\int_{\Omega} g(u(t,x))|\nabla u(t,x)|^2dx\leq 0,\,\,\,\,\, t\in (0,T),$$ in view of (H).

We simplify the second term of the last inequality in the following form
\begin{align*}\int_{\Omega} g(u(t,x))|\nabla u(t,x)|^2dx&\geq c_0\int_{\Omega} u^{m}(t,x)|\nabla u(t,x)|^2dx\\&=c_0\int_{\Omega} | u^\frac{m}{2}(t,x) \nabla u(t,x)|^2dx\\&\geq C\int_{\Omega} |\nabla u^\frac{m+2}{2}(t,x)|^2dx,\,C>0.\end{align*}
Applying the Poincar\'{e} inequality and the embedding theorem, we have
\begin{align*}\int_{\Omega} |\nabla u^\frac{m+2}{2}(t,x)|^2dx&\geq\lambda_1\int_{\Omega} |u^\frac{m+2}{2}(t,x)|^2dx\\&=\lambda_1\int_{\Omega} |u(t,x)|^{m+2}dx\\&= \lambda_1\|u(t,\cdot)\|_{L^{m+2}(\Omega)}^{m+2} \\&\geq C \|u(t, \cdot)\|_{L^{2}(\Omega)}^{m+2},\end{align*} where $C>0.$

The inequality \eqref{Lp}, the hypothesis (H1) and the above calculations give
$$\partial^{\alpha}_{0+,t}\|u(t, \cdot)\|_{L^{2}(\Omega)}+Ct^\beta\|u(t, \cdot)\|_{L^{2}(\Omega)}^{m+1}\leq 0,\,\,\, t>0.$$
Then $E(t)=\|u(t,\cdot)\|_{L^{2}(\Omega)}$ is a subsolution of the equation \eqref{ode} and from this it follows \eqref{m1}.
\end{proof}

\subsection{Time-fractional degenerate diffusion equation}

In this subsection, we are interested in case $$\mathcal{A}(u):= f(u)\Delta u,$$ where we study the time-fractional degenerate diffusion equation
\begin{equation}\label{1.1-4+}\partial^{\alpha}_{0+,t} u(t,x) - a(t) f(u(t,x))\Delta u(t,x) = 0,\,\left({t,x}
\right) \in \Omega_+.\end{equation}
Let us assume that the function $f(u)$ has the property:
\begin{description}
\item[(H2)] $f(u)u\geq c_1 u^{q+1},\, q\geq 0,\,c_1>0.$
\end{description}
\begin{theorem}\label{th4+}
Let $u$ be a solution of problem \eqref{1.1-4+}, \eqref{1.2}-\eqref{1.3}. Let $a(t)$ and $f(u)$ satisfy the hypotheses (H) and (H2).
Assume that $u_0\in L^{2}(\Omega),$ then
\begin{equation}\label{m2}\|u(t,\cdot)\|_{L^{2}(\Omega)}\leq \frac{M}{1+t^{\frac{\alpha+\beta}{q+1}}},\,t\geq 0,
\end{equation} holds true, where $M>0$ is an arbitrary constant depending on the $L^2(\Omega)$ norm of $u_0$.
\end{theorem}
\begin{proof} Multiplying the equation \eqref{1.1-4+} by $u$ and integrating by parts over the domain $\Omega$, we arrive at
$$\int_{\Omega} u(t,x) [\partial^{\alpha}_{0+,t} u(t,x)]dx+c_1\kappa t^\beta\int_{\Omega} \nabla u^{q+1}(t,x) \nabla u(t,x) dx\leq 0,\,\,\,\,\, t>0,$$ from the hypotheses (H) and (H2).

Hence, the second term of the last inequality can be written as
\begin{align*}\int_{\Omega} \nabla u^{q+1}(t,x) \nabla u(t,x) dx&=(q+1)\int_{\Omega} u^{q}(t,x)|\nabla u(t,x)|^2dx\\&=(q+1)\int_{\Omega} | u^\frac{q}{2}(t,x) \nabla u(t,x)|^2dx\\&=\frac{2(q+1)}{q+2}\int_{\Omega} |\nabla u^\frac{q+2}{2}(t,x)|^2dx.\end{align*}
Using the Poincar\'{e} inequality and the embedding theorem, we deduce that
\begin{align*}\int_{\Omega} |\nabla u^\frac{q+2}{2}(t,x)|^2dx&\geq\lambda_1\int_{\Omega} |u^\frac{q+2}{2}(t,x)|^2dx\\&=\lambda_1\int_{\Omega} |u(t,x)|^{q+2}dx\\&= \lambda_1\|u(t,\cdot)\|_{L^{q+2}(\Omega)}^{q+2} \\&\geq C \|u(t, \cdot)\|_{L^{2}(\Omega)}^{q+2},\end{align*} where $C>0.$

The inequality \eqref{Lp} and the above calculations give
$$\partial^{\alpha}_{0+,t}\|u(t, \cdot)\|_{L^{2}(\Omega)}+Ct^\beta\|u(t, \cdot)\|_{L^{2}(\Omega)}^{q+1}\leq 0,\,\,\, t>0.$$
Then $E(t)=\|u(t,\cdot)\|_{L^{2}(\Omega)}$ is a subsolution of the equation \eqref{ode} which gives \eqref{m2}.
\end{proof}

\subsection{Time-fractional mean curvature equation}
We consider the case $$\mathcal{A}(u):=\text{div}\left(\frac{\nabla u(t,x)}{1+|\nabla u(t,x)|^2}\right).$$ In this case the equation \eqref{1.1} coincides with mean curvature sub-diffusion equation
\begin{equation}\label{1.1-5}\partial^{\alpha}_{0+,t} u(t,x) - a(t)\text{div}\left(\frac{\nabla u(t,x)}{{1+|\nabla u(t,x)|^2}}\right) = 0,\,\left({t,x}
\right) \in \Omega_+.\end{equation}

\begin{theorem}\label{th5}
Assume that $u$ is the solution of problem \eqref{1.1-5}, \eqref{1.2}-\eqref{1.3}. Let $u_0\in L^2(\Omega)$ and let $a(t)$ satisfies (H). Suppose that either
\begin{equation}\label{36}
n\in \{1,2\}\,\,\,\,\,and \,\,\,\,\, \sup_{t>0}\int_{\Omega}\sqrt{1+|\nabla u(t,x)|^2}dx<+\infty,
\end{equation}
or
\begin{equation}\label{37}
\sup_{x\in \Omega, \, t>0}|\nabla u(t,x)|<+\infty,\,n\geq 3,
\end{equation}
it holds tone
\begin{equation}\label{mc1}
\|u(t,\cdot)\|_{L^2(\Omega)}\leq \frac{M}{1+t^{\alpha+\beta}},
\end{equation}
where $M>0$ is an arbitrary constant depending on the $L^2(\Omega)$ norm of $u_0$.
\end{theorem}
\begin{remark}
The problem \eqref{1.1-5}, \eqref{1.2}-\eqref{1.3} was studied in \cite{Vald2} with $\beta=0,$ and the estimate was proven
\begin{equation*}\|u(t,\cdot)\|_{L^2(\Omega)}\leq \frac{M}{1+t^{{\alpha}}},
  \end{equation*}
which coincides with the estimate \eqref{mc1} when $\beta=0$.
\end{remark}
\begin{proof}[Proof of Theorem \ref{th5}]
Multiplying the equation \eqref{1.1-5} by $u$ and integrating by parts over the domain $\Omega$, we have
$$\int_{\Omega} u(t,x) [\partial^{\alpha}_{0+,t} u(t,x)]dx+a(t)\int_{\Omega} \frac{|\nabla u(t,x)|^2}{{1+|\nabla u(t,x)|^2}} dx= 0,\,\,\,\,\, t>0.$$

According to the results of \cite{Vald2}, it is true that
\begin{equation*}
\int_{\Omega}\frac{|\nabla u(t,x)|^2}{\sqrt{1+|\nabla u(t,x)|^2}}dx\geq C||u||^2_{L^2(\Omega)},\,C>0,\end{equation*} where either $$n\in \{1,2\}\,\,\,\,\,and \,\,\,\,\, \sup_{t>0}\int_{\Omega}\sqrt{1+|\nabla u(t,x)|^2}dx<+\infty,$$ or $$\sup_{x\in \Omega, \, t>0}|\nabla u(t,x)|<+\infty,\,\,n\geq 3.$$
Then, taking into account (H) and \eqref{Lp}, we obtain
$$\partial^{\alpha}_{0+,t}\|u(t, \cdot)\|_{L^{2}(\Omega)}+Ct^\beta\|u(t, \cdot)\|_{L^{2}(\Omega)}\leq 0,\,\,\, t>0.$$
The function $E(t)=\|u(t,\cdot)\|_{L^{2}(\Omega)}$ is a subsolution of the equation \eqref{ode} which implies \eqref{mc1}.
\end{proof}

\subsection{Sub-diffusion equation with Kirchhoff nonlinearity}

In this subsection, we consider the operator $$\mathcal{A}(u):=M\left(\|\nabla u(t,\cdot)\|_{L^{q}(\Omega)}\right)\Delta_p u,$$
where $q>1,$ $p>1.$ Now, we study the time-fractional diffusion equation with Kirchhoff nonlinearity
\begin{equation}\label{1.1-4}\partial^{\alpha}_{0+,t} u(t,x) - a(t)M\left(\|\nabla u(t,\cdot)\|_{L^{q}(\Omega)}\right)\Delta_p u(t,x) = 0,\,\left({t,x}
\right) \in \Omega_+.\end{equation}

Suppose that the following hypothesis hold:\\
\begin{description}
\item[(H3)] $M\left(s\right)\geq b s^\gamma,\,s>0,\,b>0,\,\gamma\geq 0.$
\end{description}
\begin{theorem}\label{th4-1}
Let $u$ be a solution of problem \eqref{1.1-4}, \eqref{1.2}-\eqref{1.3} with initial data $u_0\in L^{2}(\Omega)$ and let $\frac{2n}{n+2}\leq p,q <\infty.$ Assume that the hypotheses (H) and (H3) holds true with $\gamma$. Then
  \begin{equation}\label{k1}\|u(t,\cdot)\|_{L^{2}(\Omega)}\leq \frac{M}{1+t^{\frac{\alpha+\beta}{\gamma+p-1}}},\,t\geq 0,
  \end{equation} where $M>0$ is an arbitrary constant depending on the $L^2(\Omega)$ norm of $u_0$.
\end{theorem}
\begin{example} Let $p=q=2$ and $M(s)=k_1+k_2 s,\,k_1, k_2>0.$ Hence, instead of the equation \eqref{1.1-4} we get the equation \begin{equation*}\partial^{\alpha}_{0+,t} u(t,x) - a(t)\left(k_1+k_2\|\nabla u(t,\cdot)\|_{L^{2}(\Omega)}\right)\Delta u(t,x) = 0.\end{equation*}
Then, by Theorem \ref{th4}, we have $$\|u(t,\cdot)\|_{L^{2}(\Omega)}\leq \frac{M}{1+t^{\frac{\alpha+\beta}{2}}}.$$ In the same way, one can consider different examples for the functions
$$M(s)=k,k=const>0,$$
$$M(s)=e^{ks}\geq 1+ks>ks,k>0,$$
$$M(s)=(1+s)^\gamma\geq 1+s>s,\gamma\geq 1,$$
as well as other functions satisfying the hypothesis (H3).
\end{example}
\begin{proof}[Proof of Theorem \ref{th4-1}] Multiplying the equation \eqref{1.1-4} by $u$ and integrating by parts over the domain $\Omega$, we obtain
$$\int_{\Omega} u(t,x)[ \partial^{\alpha}_{0+,t} u(t,x)]dx+a(t)M\left(\|\nabla u(t,\cdot)\|_{L^{q}(\Omega)}\right)\int_{\Omega} |\nabla u(t,x)|^p dx= 0,\,\,\,\,\, t>0.$$
From \eqref{Lp} and the hypotheses (H) and (H3), it follows that
$$\partial^{\alpha}_{0+,t}\|u(t, \cdot)\|_{L^{2}(\Omega)}+\kappa t^\beta \|\nabla u(t,\cdot)\|^{\gamma}_{L^{q}(\Omega)}\|\nabla u(t,\cdot)\|^{p}_{L^{p}(\Omega)}\leq 0,\,\,\,\,\, t>0.$$
Let $\frac{2n}{n+2}\leq p <\infty$ and $\frac{2n}{n+2}\leq q <\infty.$ Then, from the Sobolev inequality, it yields
$$\|\nabla u(t,\cdot)\|^{\gamma}_{L^{q}(\Omega)}\geq C_1\|u(t,\cdot)\|^{\gamma}_{L^{2}(\Omega)}$$
and
$$\|\nabla u(t,\cdot)\|^{p}_{L^{p}(\Omega)}\geq C_2\|u(t,\cdot)\|^{p}_{L^{2}(\Omega)},$$
where $C_1>0$ and $C_2>0$ are constants of Sobolev inequality.
Hence, we deduce that
$$\partial^{\alpha}_{0+,t}\|u(t, \cdot)\|_{L^{2}(\Omega)}+Ct^\beta\|u(t, \cdot)\|_{L^{2}(\Omega)}^{\gamma+p-1}\leq 0,\,\,\, t>0.$$
As $\gamma+p-1>0,$ the function $E(t)=\|u(t,\cdot)\|_{L^{2}(\Omega)}$ is a sub solution of the equation \eqref{ode} and from this follows \eqref{k1}.
\end{proof}

\section{Applications}
In this section, we present the applications of the obtained results in the above sections to the time-fractional reaction-diffusion equations with nonlinear sources.
\subsection{Time-fractional Fisher-KPP-type equation}
We consider the time-fractional Fisher-KPP-type equation
\begin{equation}\label{a1}\partial^{\alpha}_{0+,t}v-t^\beta\Delta v=-v(1-v),\, t>0, x\in\Omega, 
\end{equation} supplemented with the Dirichlet boundary condition \eqref{1.3} and with the initial condition \eqref{1.2}. The equation \eqref{a1} can be obtained from the fractional Fischer-KPP equation $$\partial^{\alpha}_{0+,t}\tilde{v}-t^\beta\Delta \tilde{v}=\tilde{v}(1-\tilde{v}),$$ by transforming $\tilde{v}\longmapsto 1-v$.

Assume that, there is a global solution $v$ to problem \eqref{a1}, \eqref{1.2}, \eqref{1.3}. Then, repeating the proof technique from \cite{Kir2}, one can show that for $0<v(0,x)\leq 1,\,x\in \Omega,$ the global solution of problem \eqref{a1}, \eqref{1.2}, \eqref{1.3} satisfies
$$0< v(t,x)\leq 1,\,\,\, t>0,\, x\in \Omega.$$
We will study on the asymptotic behavior of the global solution for large times. Suppose that $0<v(0,x)\leq 1,\,x\in \Omega.$ Since $0< v(t,x)\leq 1,$ then we have $$-v(t,x) + v^2(t,x) \leq 0,\,\,\,\textit{for all}\,\,(t,x)\in (0,+\infty)\times\Omega.$$ Hence, $v$ satisfies the linear time-fractional diffusion inequality
$$\partial^{\alpha}_{0+,t}v-t^\beta\Delta v \leq 0,\, t>0, x\in\Omega,$$ with the Cauchy-Dirichlet conditions \eqref{1.2}, \eqref{1.3}.
Then, in view of $v$ is a sub solution of problem \eqref{1.1-1}, \eqref{1.2}, \eqref{1.3}, by the comparison principle (see \cite{Zacher2}), we obtain $$v(t,x)\leq u(t,x),\,t>0,\,x\in\Omega,$$ where $u$ is the solution of problem \eqref{1.1-1}, \eqref{1.2}, \eqref{1.3}. Consequently, from the results of Theorem \ref{th1} it follows that
\begin{align*}\|v(t,\cdot)\|_{L^2(\Omega)}\leq\frac{M_1}{1+\lambda_1t^{\alpha+\beta}},\,t> 0,\end{align*} where $M_1$ is a positive constant.

\subsection{Semilinear fractional porous medium equation}
Let us consider the time-fractional porous medium equation
\begin{equation}\label{pme1}
\partial^{\alpha}_{0+,t}w-t^\beta\Delta |w|^{m}w +\mu |w|^{p}w=0,\,(t,x)\in\Omega_+,
\end{equation}
with the Cauchy-Dirichlet conditions \eqref{1.2}, \eqref{1.3}. Here $\mu\in\mathbb{R},$ $m\geq 0$ and $p>1.$

Assume that there is a global positive solution $w>0$ to problem \eqref{pme1}, \eqref{1.2}, \eqref{1.3}. Then the equation \eqref{pme1} can be rewritten as
\begin{equation*}
\partial^{\alpha}_{0+,t}w-t^\beta\Delta w^{m} =-\mu w^{p},\,\,(t,x)\in\Omega_+,
\end{equation*} or, for $\mu>0$ in the form of a fractional differential inequality
\begin{equation*}
\partial^{\alpha}_{0+,t}w-t^\beta\Delta w^{m} \leq 0,\,(t,x)\in\Omega_+.
\end{equation*}
Hence, $w$ is a subsolution of problem \eqref{1.1-1}, \eqref{1.2}, \eqref{1.3} and
$$\|w(t,\cdot)\|_{L^{2}(\Omega)}\leq \frac{M_2}{1+t^{\frac{\alpha+\beta}{m+1}}},$$ where $M_2$ is a positive constant.

In the same way, one can show the decay estimates for more general evolution equations with non-negative nonlinear terms $f(s)\geq 0,\,s>0$, i.e.
$$\partial^{\alpha}_{0+,t} u(t,x) - a(t)\mathcal{A}(u(t,x)) + f(u(t,x))=0,\,\left({t,x}
\right) \in \Omega_+,$$
using the estimates we obtained in the previous sections. Here $\mathcal{A}$ one of the operators: the Laplacian, the porous medium operator, degenerate operator, p-Laplacian, mean curvature operator, Kirchhoff operator.

\section*{Conclusion and future perspectives}
The main object of the paper is to study the Cauchy-Dirichlet problem for the time-fractional evolution equation
$$\partial^{\alpha}_{0+,t} u(t,x) - a(t)\mathcal{A}(u(t,x)) =0,\,\left({t,x}
\right) \in \Omega_+.$$
Here we have assumed (except case $\alpha=1$) that $a(t)$ is bounded from below, i.e. $$a(t)\geq\kappa t^\beta,\,\kappa>0,\,\beta>-\alpha.$$
For some types of operator $\mathcal{A}$ such as: \emph{the Laplacian, the porous medium operator, degenerate operator, p-Laplacian, mean curvature operator, and Kirchhoff operator,} the decay estimates of solutions have been obtained.

As a continuation of this paper, we plan to study the equation \eqref{1.1}, on the whole Euclidean space $\mathbb{R}^n$, and intend to obtain the decay estimates of solutions.

Naturally, there are still unexplored objects in this direction. For example, it would be interesting to study the equation \eqref{1.1} with other degeneracies $a(t)$ than the hypothesis (H), for example when $a(t)$ represent a periodic function. The question of the optimality of the decay rates also remains open in general. In some cases (p-Laplacian, Porous medium), it can be proved by minimizing functionals, as in \cite{Lind}. This approach was previously applied in \cite{Zacher1} for the nondegenerate case. It seems interesting to extend the results obtained in this paper to other fractional derivatives, such as the distributed order fractional derivative, the fractional derivative with Sonine kernel, etc.

\section*{Declaration of competing interest}
	The Authors declares that there is no conflict of interest

\section*{Authors’ Contributions} A.S.: writing-original draft; B.T.: supervision, writing-review and editing. All authors gave final approval for publication and agreed to be held accountable for the work performed
therein.

\section*{Data Availability} Data sharing not applicable to this article as no datasets were generated or analysed during the current study

\section*{Acknowledgments}
This research has been funded by the Science Committee of the Ministry of Education and Science of the Republic of Kazakhstan (Grant No. AP09259578). BT is also supported by the FWO Odysseus 1 grant G.0H94.18N: Analysis and Partial Differential Equations and by the Methusalem programme of the Ghent University Special Research Fund (BOF) (Grant number 01M01021).

The authors would like to thank to the reviewers for their valuable comments
and remarks.

No new data was collected or generated during the course of research.


\begin{thebibliography}{99}
\bibitem{Vald1} E. Affili, E. Valdinoci, Decay estimates for evolution equations with classical and fractional time-derivatives. {\it J. Differential Equations}, 266:7 (2019), 4027--4060.
\bibitem{Caff} M. Allen, L. Caffarelli, A. Vasseur, A parabolic problem with a fractional time derivative. {\it Arch. Ration. Mech. Anal.}, 221:2 (2016), 603--630.
\bibitem{Kir1} A. Alsaedi, B. Ahmad, M. Kirane, Maximum principle for certain generalized time and space fractional diffusion equations. {\it Quart. Appl. Math.}, 73:1 (2015), 163--175.
\bibitem{Kir2} A. Alsaedi, M. Kirane, B. T. Torebek, Global existence and blow-up for a space and time nonlocal reaction-diffusion equation. {\it Quaest. Math.} 44:4 (2021), 747--753.
\bibitem{Bol} M. Bologna, A. Svenkeson, B. J. West, P. Grigolini, Diffusion in heterogeneous media: An iterative scheme for finding approximate solutions to fractional differential equations with time-dependent coefficients, {\it J. Comp. Phys.} 293 (2015), 297--311.
\bibitem{TSim19} L. Boudabsa, T. Simon, Some Properties of the Kilbas-Saigo Function, {\it Mathematics.} 9:3 (2021). 1--24. doi:10.3390/math9030217
\bibitem{Cos} F. S. Costa, E. C. de Oliveira, A. R. G. Plata, Fractional diffusion with time-dependent diffusion coefficient, {\it Rep. Math. Phys.} 87 (2021), 59--79.
\bibitem{Vald4} S. Dipierro, X. Ros-Oton, E. Valdinoci, Nonlocal problems with Neumann boundary conditions. {\it Rev. Mat. Iberoam.}, 33:2 (2017), 377--416.
\bibitem{Vald3} S. Dipierro, E. Valdinoci, A simple mathematical model inspired by the Purkinje cells: from delayed travelling waves to fractional diffusion. {\it Bull. Math. Biol.} 80:7 (2018), 1849--1870.
\bibitem{Vald2} S. Dipierro, E. Valdinoci, V. Vespri, Decay estimates for evolutionary equations with fractional time-diffusion. {\it J. Evol. Equ.}, 19:2 (2019), 435--462.
\bibitem{Nieto} J.-D. Djida, J. J. Nieto, I. Area, Nonlocal time porous medium equation with fractional time derivative. {\it Rev. Mat. Complut.} 32:2 (2019), 273--304.
\bibitem{Dong} H. Dong, D. Kim, Time fractional parabolic equations with measurable coefficients and embeddings for fractional parabolic Sobolev spaces. {\it Int. Math. Res. Not. IMRN}, 22 (2021), 17563--17610.
\bibitem{FaLe} K. S. Fa, E. K. Lenzi, Time-fractional diffusion equation with time dependent diffusion coefficient, {\it Phys. Rev. E}, 72 (2005), 011107.
\bibitem{Hilfer} R. Hilfer, Fractional time evolution, in Applications of Fractional Calculus in Physics, World Science Publishing, River Edge, NJ, 2000.
\bibitem{Hris} J. Hristov, Subdiffusion model with time-dependent diffusion coefficient: Integral-balance solution and analysis, {\it Thermal Science}, 21 (2017), 69--80.
\bibitem{Lind} P. Juutinen, P. Lindqvist, Pointwise decay for the solutions of degenerate and singular parabolic equations, {\it Adv. Differential Equations}, 14 (2009), 663--684.
\bibitem{KS95} A. A. Kilbas, M. Saigo, On the solution of integral equations of Abel-Volterra type. {\it Differential Integral Equations.} 8 (1995), 993--1011.
\bibitem{1} A. A. Kilbas, H. M. Srivastava, J. J. Trujillo. \emph{Theory and Applications of Fractional Differential Equations.} Elsevier. North-Holland. Mathematics studies. 2006. -539p.
\bibitem{Yam1} A. Kubica, M. Yamamoto, Initial-boundary value problems for fractional diffusion equations with time-dependent coefficients. {\it Fract. Calc. Appl. Anal.} 21:2 (2018), 276--311.
\bibitem{Luch1} Z. Li, Y. Luchko, M. Yamamoto, Asymptotic estimates of solutions to initial-boundary-value problems for distributed order time-fractional diffusion equations. {\it Fract. Calc. Appl. Anal.} 17:4 (2014), 1114--1136.
\bibitem{Yam2} Z. Li, Y. Liu, M. Yamamoto, Initial-boundary value problems for multi-term time-fractional diffusion equations with positive constant coefficients, {\it Appl. Math. Comput.} 257 (2015) 381--397.
\bibitem{Luch2} Y. Luchko, Maximum principle for the generalized time-fractional diffusion equation. {\it J. Math. Anal. Appl.} 351:1 (2009), 218--223.
\bibitem{Main} F. Mainardi, {\it Fractional calculus and waves in linear viscoelasticity. An introduction to mathematical models.} Imperial College Press, London, 2010.
\bibitem{Metz1} R. Metzler, J. Klafter, The random walk's guide to anomalous diffusion: A fractional dynamics approach, {\it Phys. Rep.}, 339 (2000), 1--77.
\bibitem{M-L03} G. M. Mittag-Leffler, Sur la nouvelle fonction $E_\alpha(x)$. {\it C.R. Acad. Sci. Paris.} 137 (1903), 554--558.
\bibitem{Nig} R. R. Nigmatullin, The realization of the generalized transfer equation in a medium with fractal geometry. {\it Phys. Stat. Sol.} 133 (1986), 299--318.
\bibitem{Prus} J. Pr\"{u}ss, {\it Evolutionary Integral Equations and Applications}, Monogr. Math. 87, Birkh\"{a}user, Basel, 1993.
\bibitem{Tar} V. E. Tarasov, Review of some promising fractional physical models, {\it Internat. J. Modern Phys. B}, 27:9 (2013), 1330005.
\bibitem{Uch} V. V. Uchaikin, Montroll-Weiss problem, fractional equations, and stable distributions. {\it Internat. J. Theoret. Phys.} 39:8 (2000), 2087--2105.
\bibitem{Zacher} V. Vergara, R. Zacher, A priori bounds for degenerate and singular evolutionary partial integro-differential equations. {\it Nonlinear Anal.} 73:11 (2010), 3572--3585.
\bibitem{Zacher1} V. Vergara, R. Zacher, Optimal decay estimates for time-fractional and other nonlocal subdiffusion equations via energy methods. {\it SIAM J. Math. Anal.}, 47:1 (2015), 210--239.
\bibitem{Zacher2} V. Vergara, R. Zacher, Stability, instability and blowup for time fractional and other nonlocal in time semilinear subdiffusion equations, {\it J. Evol. Equ.}, 17 (2017), 599--626.
\bibitem{W05} A. Wiman, \"{U}berden fundamentalsatz in der theorie der funktionen $E_\alpha(x)$. {\it Acta Math}. 29 (1905), 191--201.
\bibitem{Zacher3} R. Zacher, Boundedness of weak solutions to evolutionary partial integro-differential equations with discontinuous coefficients. {\it J. Math. Anal. Appl.} 348:1 (2008), 137--149.
\end{thebibliography}
\end{document}